\documentclass[11pt]{article}
\usepackage[latin1]{inputenc}
\usepackage{amsmath,amsthm,amssymb}
\usepackage{amsfonts}
\usepackage{amsmath,amsthm,amssymb,amscd}
\usepackage{latexsym}
\usepackage{color}
\usepackage{graphicx}
\usepackage{mathrsfs}
 \usepackage{cite}

\makeatletter \@addtoreset{equation}{section} \makeatother

\newtheorem{theorem}{Theorem}[section]

\newtheorem{proposition}{Proposition}[section]
\newtheorem{lemma}{Lemma}[section]

\begin{document}

\title{\sc On the \emph{p}-fractional Schr\"{o}dinger-Kirchhoff equations 
with electromagnetic fields and the Hardy-Littlewood-Sobolev nonlinearity}

\author{\sc {Min Zhao$^{a}$,   Yueqiang
Song$^{a}$\thanks{{{Corresponding author}}} and Du\v{s}an D.
Repov\v{s}$^{b}$}\thanks{{{{\it E-mail addresses:} ywx7529@163.com
(M. Zhao), songyq16@mails.jlu.edu.cn (Y. Song),
dusan.repovs@guest.arnes.si (D.D. Repov\v{s})}}}\\
$^{\small\mbox{a}}${\small College of Mathematics, Changchun Normal
University,  Changchun, 130032,  P.R. China}\\[-0.2cm]
$^{\small\mbox{b}}$ {\small Faculty of Education and Faculty of
Mathematics and Physics, University of Ljubljana,} \\[-0.2cm]
{\small  \& Institute of Mathematics,
Physics and Mechanics, Ljubljana,  1000, Slovenia}}
\date{}
\maketitle

  \begin{abstract}
{In this article, we deal with the following $p$-fractional
Schr\"{o}dinger-Kirchhoff equations with electromagnetic fields and
the Hardy-Littlewood-Sobolev nonlinearity:
$$
M\left([u]_{s,A}^{p}\right)(-\Delta)_{p, A}^{s} u+V(x)|u|^{p-2}
u=\lambda\left(\int_{\mathbb{R}^{N}} \frac{|u|^{p_{\mu,
s}^{*}}}{|x-y|^{\mu}} \mathrm{d}y\right)|u|^{p_{\mu, s}^{*}-2}
u+k|u|^{q-2}u,\ x \in \mathbb{R}^{N},$$ where  $0<s<1<\emph{p}$, $ps
< N$, $\emph{p}<\emph{q}<2p^{*}_{s,\mu}$, $0<\mu<N$, $\lambda$ and
$k$ are some  positive parameters,
$\emph{p}^{*}_{s,\mu}=\frac{\emph{p}N-\emph{p}\frac{\mu}{2}}{N-\emph{p}s}$
is the critical exponent with respect to the
Hardy-Littlewood-Sobolev inequality, and functions $V$, $M$ satisfy
the suitable conditions. By proving  the compactness results with
the help of the fractional version of concentration compactness
principle, we establish the existence of nontrivial solutions to
this problem.}

\medskip

 \medskip

 \medskip
\end{abstract} 

\maketitle

\section{Introduction}\label{s1}
In this article,  we  intend to study the following
\emph{p}-fractional Schr\"{o}dinger-Kirchhoff equations with
electromagnetic fields and the Hardy-Littlewood-Sobolev nonlinearity
in $\mathbb{R}^{N}$
\begin{align}\label{e1.1}
M\left([u]_{s,A}^{p}\right)(-\Delta)_{p, A}^{s} u+V(x)|u|^{p-2}
u=\lambda\left(\int_{\mathbb{R}^{N}} \frac{|u|^{p_{\mu,
s}^{*}}}{|x-y|^{\mu}} \mathrm{d} y\right)|u|^{p_{\mu, s}^{*}-2}
u+k|u|^{q-2}u,\quad  x \in \mathbb{R}^{N},
\end{align}
where $0<s<1<\emph{p}$, $ps < N$,
$\emph{p}<\emph{q}<2p^{*}_{s,\mu}$, $0<\mu<N$, $\lambda$ and $k$ are
some  positive parameters,
\begin{equation*}
[u]_{s, A}^{p}:=\iint_{\mathbb{R}^{2 N}}
\frac{\left|u(x)-\mathrm{e}^{i(x-y) \cdot A(\frac{x+y}{p})}
u(y)\right|^{p}}{|x-y|^{N+p s}}\mathrm{~d}x\mathrm{~d}y,
\end{equation*}
$\emph{p}^{*}_{s,\mu}=\frac{\emph{p}N-\emph{p}\frac{\mu}{2}}{N-\emph{p}s}$
is the critical exponent with respect to the
Hardy-Littlewood-Sobolev inequality, $V\in C(\mathbb{R}^{N},
\mathbb{R}_{0}^{+})$ is an electric potential, $A\in
C(\mathbb{R}^{N}, \mathbb{R}^{N})$ is a magnetic potential, and $V$,
$M$ satisfy the following assumptions
\begin{itemize}
\item[$(V)$] $V:\mathbb{R}^{N} \rightarrow \mathbb{R}$ is a continuous function and has critical frequency, that is, $V(0) = \min
_{x\in\mathbb{R}^{N}}V(x)=0$. Moreover, the set
$V^{\tau_{0}}=\left\{x\in\mathbb{R}^{N}:V(x) <\tau_{0}\right\}$ has
finite Lebesgue measure for some $\tau_{0}>0$.

\item[$(M)$] $\left(m_{1}\right)$ The Kirchhoff function $M:\mathbb{R}_{0}^{+}\rightarrow\mathbb{R}^{+}$ is a  continuous and nondecreasing.
In addition, there exists a positive constant  $m_{0}>0$ such that
$M(t)\geq m_{0}$ for all $t\in\mathbb{R}_{0}^{+}$;

$\left(m_{2}\right)$ For some  $\sigma\in$ $\left(p /q, 1\right],$
we have $\widetilde{M}(t)\geq\sigma M(t)t$ for all $t\geq0$, where
$\widetilde{M}(t)=\int_{0}^{t}M(s)ds$.
\end{itemize}

When $p=2$, we know that  the fractional operator $(-\Delta)_{A}^s$,
which up to normalization constants, can be defined on smooth
functions $u$ as
\begin{eqnarray*}
(-\Delta)_{A}^s u(x) := 2\lim_{\varepsilon \rightarrow 0}
\int_{\mathbb{R}^N \setminus
B_\varepsilon(x)}\frac{u(x)-e^{i(x-y)\cdot
A(\frac{x+y}{2})}u(y)}{|x-y|^{N+2s}}dy, \quad x\in  \mathbb{R}^N,
\end{eqnarray*}
see d'Avenia and Squassina \cite{da1}. There already exist  several
papers dedicated to the study of the Choquard equation, this problem
can be used to describe many physical models \cite{la, pek}.
Recently, d'Avenia and Squassina \cite{da1} considered the following
fractional Choquard equation of the form
\begin{equation}\label{e1.2}
(-\Delta)^su+ \omega u =
\left(\mathcal{K}_\mu*|u|^{p}\right)|u|^{p-2}u, \quad u \in
H^s(\mathbb{R}^N), \quad N\geq 3,
\end{equation}
and the existence of ground state solutions was obtained by using
the Mountain pass theorem and the Ekeland variational principle. For
more results on
 problems with the Hardy-Littlewood-Sobolev
nonlinearity without the magnetic operator case, see  \cite{cas, ma,
pa2, ch,  so1, so2}.

For the case $p \neq 2$,   Iannizzotto  et al. \cite{a11}
investigated the  following fractional \emph{p}-Laplacian equation
\begin{eqnarray}\label{e1.3}
\left\{\begin{array}{lll}
(-\Delta)_{p}^{s}u=f(x,u)  \quad &\text {in}\quad    \Omega,\\
  u=0             \quad &\text {in}\quad    \mathbb{R}^{N}\backslash\Omega.
\end{array}\right.
\end{eqnarray}
The existence and multiple solutions for problem \eqref{e1.3} was
proved by using the Morse theory.   Xiang et al. \cite{a12} dealt
with a class of Kirchhoff-type problems driven by nonlocal elliptic
integro-differential operators, and two existence theorems were
obtained with the help of the variational method.  Souza \cite{a13}
studied a class of nonhomogeneous fractional quasilinear equations
in $\mathbb{R}^{N}$ with exponential growth of the form
\begin{equation}
\begin{aligned}\label{4''}
(-\bigtriangleup)_{p}^{s}u+V(x)|u|^{p-2}u=f(x,u)+\lambda h \quad
\text {in}\ \Omega.
\end{aligned}
\end{equation}
By using a suitable Trudinger-Moser inequality for fractional
Sobolev spaces, they  established the existence of weak solutions
for  problem \eqref{4''}. In particular, Nyamoradi and Razani
\cite{ny} considered a class of new Kirchhoff-type equations
involving the fractional $p$-Laplacian and Hardy-Littlewood-Sobolev
critical nonlinearity. The existence of infinitely many solutions
was obtained   by using the concentration compactness principle and
Krasnoselskii's genus theory.  For more recent advances on this kind
of problems, we refer the readers to \cite{a6, a14, a27, Li1, Li2,
LS, liang5, liang3, a15, a21, mu1, x1, zhang2, ho, pa}.

On the other hand, one of the main features of problem \eqref{e1.1}
is the presence of the magnetic field  operator $A$. When $A \neq
0$, some authors have studied the following equation
\begin{equation}\label{e1.5}
-(\nabla u-{\rm i} A)^2u + V(x)u = f(x, |u|)u,
\end{equation}
which has appeared in recent years, where the  magnetic operator in
equation \eqref{e1.5} is given by
\begin{displaymath}
-(\nabla u-{\rm i} A)^2u = -\Delta u +2iA(x)\cdot\nabla u +
|A(x)|^2u + iu \mbox{div} A(x).
\end{displaymath}
Squassina and Volzone \cite{sq1} proved that up to correcting the
operator by the factor $(1-s)$, it follows that $(-\Delta)^s_A u$
converges to $-(\nabla u-{\rm i} A)^2u$ as $s\rightarrow1$. Thus, up
to normalization, the nonlocal case can be seen as an approximation
of the local one.

Recently, many researchers have paid attention to the problems  with
fractional magnetic operator. In particular, Mingqi et al.
\cite{MPSZ}  proved some existence results for the following
Schr$\ddot{\mbox{o}}$dinger--Kirchhoff type equation involving the
magnetic operator
\begin{equation}\label{e1.6}
M([u]_{s,A}^2)(-\Delta)_A^su+V(x)u=f(x,|u|)u\quad \text{in
$\mathbb{R}^N$},
\end{equation}
where $f$ satisfies the subcritical growth condition. For the
critical growth case, Binlin et al.  \cite{a9} considered the
following fractional Schr\"{o}dinger-equation with critical
frequency and critical growth
\begin{equation}\begin{aligned}\label{5'}
 \varepsilon^{2s}(-\vartriangle)^{s}_{A_{\varepsilon}}u+V(x)u=f(x,|u|)u+K(x)|u|^{2_{\alpha}^{*}-2}u\quad \text {in}\quad \mathbb{R}^{N}.
 \end{aligned}
 \end{equation}
The existence of ground state solution tending to trivial solution
as $\varepsilon\rightarrow0$ was obtained by using variational
method. Furthermore,  Song and Shi \cite{a24} were concerned with a
class of the \emph{p}-fractional Schr\"{o}dinger-Kirchhoff equations
with electromagnetic fields, under suitable additional assumptions,
the existence of infinite solutions was obtained by using the
variational method. More results about fractional equations
involving the Hardy-Littlewood-Sobolev and critical nonlinear and
can be found in \cite{liang4, WX, xia,  zhang3}.

Inspired by the above works,  in this paper, we are interested in
the  \emph{p}-fractional Schr\"{o}dinger-Kirchhoff equations with
electromagnetic fields and the Hardy-Littlewood-Sobolev
nonlinearity. As far as we know, there have not been any results for
problem \eqref{e1.1} yet. We note that there are many difficulties
in dealing with such problems due to the presence of the
electromagnetic field and critical nonlinearity. In order to
overcome these difficulties, we shall adopt the
concentration-compactness principles and some new techniques to
prove the $(PS)_{c}$ condition. Moreover, we shall use the variational
methods in order to establish the existence and multiplicity of
solutions for problem \eqref{e1.1}. Here are our main results.
\begin{theorem}\label{the1.1}
Suppose that conditions $(V)$ and $(M)$  are satisfied. Then  there exists
$\lambda^\ast > 0$ such that if $\lambda
> \lambda^\ast >0$, then there exists
at least one solution $u_{\lambda}$ of problem \eqref{e1.1} and
$u_{\lambda}\rightarrow 0$ as $\lambda \rightarrow \infty$.
\end{theorem}

\begin{theorem}\label{the1.2}
Suppose that conditions $(V)$ and $(M)$  are satisfied. Then   for any $m\in
\mathbb{N}$, there exists $\lambda_{m}^\ast>0$ such that if $\lambda >
\lambda_{m}^\ast$, then problem \eqref{e1.1} has at least $m$ pairs of
solutions $u_{\lambda,i},u_{\lambda,-i},i=1,2,\cdots,m$ and
$u_{\lambda,\pm i} \rightarrow 0$ as $\lambda \rightarrow \infty$.
\end{theorem}

This paper is organized as follows. In Section \ref{s2}, we present the
working space and the necessary preliminaries. In Section \ref{s3}, we
apply the principle of concentration compactness to prove that the
$(PS)_{c}$ condition holds. In Section \ref{s4}, we check that the mountain
pass geometry is established. In Section \ref{s5}, we use the critical point
theory and some subtle estimates to prove our main results.

\section{Preliminaries}\label{s2}
In this section we shall give the relevant notations  and  some
useful auxiliary lemmas. 
For other background information we refer to Papageorgiou et al. \cite{PRR}.
Let
$$
W_{A}^{s, p}\left(\mathbb{R}^{N}, \mathbb{C}\right)=\left\{u \in
L^p\left(\mathbb{R}^{N}, \mathbb{C}\right):[u]_{s,
A}<\infty\right\},
$$
where $s\in (0,1)$ and
$$[u]_{s, A}:=\left(\iint_{\mathbb{R}^{2 N}} \frac{\left|u(x)-\mathrm{e}^{i(x-y) \cdot A(\frac{x+y}{p})} u(y)\right|^{p}}{|x-y|^{N+ps}} \mathrm{~d}x \mathrm{~d} y\right)^{1/p}.$$
The norm of the fractional  Sobolev space is given by
$$
\|u\|_{W_{A}^{s, p}\left(\mathbb{R}^{N},
\mathbb{C}\right)}=\left([u]_{s, A}^{p}+\|u\|_{L^{p}}^{p}\right)^{1
/ p}.
$$
In order to study our problem \eqref{e1.1},  we shall use the
following subspace of $W_{A}^{s,p}(\mathbb{R}, \mathbb{C})$ defined
by
$$
E=\left\{u \in W_{A}^{s, p}\left(\mathbb{R}^{N}, \mathbb{C}\right): \int_{\mathbb{R}^{N}} V(x)|u|^{p} \mathrm{~d} x<\infty\right\},
$$
with the norm
$$
\|u\|_{E}:=\left([u]_{s, A}^{p}+\int_{\mathbb{R}^{N}} V(x)|u|^{p}
\mathrm{~d} x\right)^{1 / p}.
$$
The condition $(V)$ implies that $E\hookrightarrow
W_{A}^{s,p}(\mathbb{R}^{N}, \mathbb{C})$ is continuous.

Next, we state the well known Hardy-Littlewood-Sobolev inequality
and the diamagnetic inequality which will be used frequently.
\begin{proposition}\label{pro2.1}
(Hardy-Littlewood-Sobolev inequality  \cite[ Theorem 4.3]{lie})
 Let
$1<t$, $r<\infty$ and $0<\mu<N$ with
$\frac{1}{r}+\frac{1}{t}+\frac{\mu}{N}=2, u \in
L^{t}\left(\mathbb{R}^{N}\right)$ and $v \in
L^{r}\left(\mathbb{R}^{N}\right)$. Then there exists a sharp
constant $C(N, \mu, t, r)>0$, independent of $u, v$, such that
$$
\int_{\mathbb{R}^{N}} \frac{|u(x)||v(y)|}{|x-y|^{\mu}} \mathrm{d} x
\mathrm{~d} y \leq C(N, \mu, t, r)\|u\|_{t}\|v\|_{r}.
$$
\end{proposition}
By the Hardy-Littlewood-Sobolev inequality, there exists $\widehat{C}(N, \mu)>0$ such that
$$
\iint_{\mathbb{R}^{2 N}} \frac{|u(x)|^{p_{\mu,
s}^{*}}|u(y)|^{p_{\mu, s}^{*}}}{|x-y|^{\mu}} \mathrm{d} x
\mathrm{~d} y \leq \widehat{C}(N, \mu)\|u\|_{p_{s}^{*}}^{2 p_{\mu,
s}} \ \ \hbox{for all} \ \ u \in E .
$$
Also, there exists $C(N, \mu)>0$ such that
$$
\iint_{\mathbb{R}^{2 N}} \frac{|u(x)|^{p_{\mu,
s}^{*}}|u(y)|^{p_{\mu, s}^{*}}}{|x-y|^{\mu}} \mathrm{d} x
\mathrm{~d} y \leq C(N, \mu)\|u\|_E^{2 p_{\mu, s}^{*}}\ \ \hbox{for
all} \ \ u \in E .
$$

\begin{lemma}\label{lem2.1}(Diamagnetic
inequality \cite[Lemma 3.1, Remark 3.2]{da1}) For every $u \in
W_{A}^{s, p}\left(\mathbb{R}^{N}, \mathbb{C}\right)$, we get $|u|
\in W^{s, p}\left(\mathbb{R}^{N}\right).$ More precisely, we have
$[|u|]_{s} \leq[u]_{s, A}.$
\end{lemma}

\section{The Palais-Smale condition}\label{s3}
First, we define the set
$$
C_{c}\left(\mathbb{R}^{N}\right)=\left\{u \in
C\left(\mathbb{R}^{N}\right): \operatorname{supp}(u) \text { is a
compact subset of } \mathbb{R}^{N}\right\}
$$
and denote by $C_{0}\left(\mathbb{R}^{N}\right)$ the closure of
$C_{c}\left(\mathbb{R}^{N}\right)$ with respect to the norm
$|\eta|_{\infty}=$ $\sup _{x \in \mathbb{R}^{N}}|\eta(x)|.$ The
measure $\mu$  gives the norm
$$
\|\mu\|=\sup _{\eta \in C_{0}\left(\mathbb{R}^{N}\right),
\quad|\eta|_{\infty}=1}|(\mu, \eta)|,
$$
where $(\mu, \eta)=\int_{\mathbb{R}^{N}} \eta \mathrm{d} \mu$.

In order to prove the compactness condition, we  introduce the
following fractional version of  the  concentration compactness
principle.
\begin{lemma}\label{lem3.1} (see Xiang and Zhang \cite{a3})
Assume that there exist bounded non-negative measures $\omega,
\zeta$ and $\nu$ on $\mathbb{R}^{N}$, and  at most countable set
$\left\{x_{i}\right\}_{i \in I} \in \Omega \backslash\{0\}$ such
that
$$
\begin{aligned}
&u_{n} \rightarrow u \text { weakly in } W^{s, p}\left(\mathbb{R}^{N}\right), \\
&\int_{\mathbb{R}^{N}} \frac{\left||u_{n}(x)|-|u_{n}(y)|\right|^{p}}{|x-y|^{N+s p}} \mathrm{~d} y \rightarrow \omega \text { weakly } * \text { in } \mathcal{M}\left(\mathbb{R}^{N}\right),\\
&\left(\int_{\mathbb{R}^{N}} \frac{|u_{n}|^{p_{\mu,
s}^{*}}}{|x-y|^{\mu}} \mathrm{d} y\right)|u_{n}|^{p_{\mu, s}^{*}}
\rightarrow \nu  \text { weakly}* \text { in }
\mathcal{M}\left(\mathbb{R}^{N}\right).
\end{aligned}
$$
Then there exist a countable sequence of points $\{x_{i}\} \subset
\mathbb{R}^{N}$ and families of positive numbers $\{\nu_i : i \in
I\}$, $\{\zeta_i : i \in I\}$ and $\{\omega_i : i \in I\}$ such that
\begin{align*}
&\omega \geq \int_{\mathbb{R}^{N}} \frac{||u(x)|-|u(y)||^{p}}{|x-y|^{N+p s}} \mathrm{~d} x+\sum_{i \in I} \omega_{i}\delta_{x_{i}},\\
&\zeta = |u|^{p_{s}^{*}}+\sum_{i \in I} \zeta_{i} \delta_{x_{i}},\\
&\nu=\left(\int_{\mathbb{R}^{N}} \frac{|u|^{p_{\mu,
s}^{*}}}{|x-y|^{\mu}} \mathrm{d} y\right)|u|^{p_{\mu,
s}^{*}}+\sum_{i \in I} \nu_{i} \delta_{x_{i}},
\end{align*}
where $I$ is at most countable. Furthermore, we have
\begin{align}\label{e3.1}
S_{p, H} \nu_{i}^{\frac{p}{2p^{*}_{\mu,s}}}\leq \omega_{i}, \quad
\nu_{i} \leq C(N, \mu) \zeta_{i}^{\frac{2 N-\mu}{N}},
\end{align}
where is the Dirac mass of mass 1 concentrated at $\{x_{i}\} \subset
\mathbb{R}^{N}$.
\end{lemma}

\begin{lemma}\label{lem3.2}(see Xiang and Zhang \cite{a3})
 Let $\left\{u_{n}\right\}_{n} \subset W^{s,
p}\left(\mathbb{R}^{N}\right)$ be a bounded sequence such that
$$
\begin{aligned}
&u_{n} \rightarrow u \text { weakly in } W^{s, p}\left(\mathbb{R}^{N}\right), \\
&\int_{\mathbb{R}^{N}} \frac{\left||u_{n}(x)|-|u_{n}(y)|\right|^{p}}{|x-y|^{N+s p}} \mathrm{~d} y \rightarrow \omega \text { weakly } * \text { in } \mathcal{M}\left(\mathbb{R}^{N}\right),\\
&\left(\int_{\mathbb{R}^{N}} \frac{|u_{n}|^{p_{\mu,
s}^{*}}}{|x-y|^{\mu}} \mathrm{d} y\right)|u_{n}|^{p_{\mu, s}^{*}}
\rightarrow \nu  \text { weakly}* \text { in }
\mathcal{M}\left(\mathbb{R}^{N}\right)
\end{aligned}
$$
and define
\begin{eqnarray*}
\omega_\infty = \lim_{R\rightarrow\infty}\limsup_{n\rightarrow\infty}\iint_{\mathbb{R}^{2 N}} \frac{\left||u_n(x)|-|u_n(y)|\right|^{p}}{|x-y|^{N+ps}} \mathrm{d}x\mathrm{d}y,
\end{eqnarray*}
\begin{eqnarray*}
\zeta_\infty =
\lim_{R\rightarrow\infty}\limsup_{n\rightarrow\infty}\int_{\mathbb{R}^{N}}|u_{n}|^{p_{s}^{*}}\mathrm{d}x,
\end{eqnarray*}
\begin{eqnarray*}
\nu_\infty = \lim_{R\rightarrow\infty}\limsup_{n\rightarrow\infty}\iint_{\mathbb{R}^{2N}} \frac{|u_{n}(x)|^{p_{\mu, s}^{*}}|u_{n}(y)|^{p_{\mu, s}^{*}}}{|x-y|^{\mu}} \mathrm{d} x \mathrm{d}y.
\end{eqnarray*}
Then the quantities $\omega_{\infty}$ ,$\zeta_{\infty}$and $\nu_{\infty}$ are well defined and satisfy
\begin{eqnarray*}
\limsup\limits\limits_{n\rightarrow\infty}\iint_{\mathbb{R}^{2N}} \frac{|u_{n}(x)|^{p_{\mu, s}^{*}}|u_{n}(y)|^{p_{\mu, s}^{*}}}{|x-y|^{\mu}} \mathrm{d} x \mathrm{d}y
= \displaystyle\int_{\mathbb{R}^N}d\nu + \nu_\infty,
\end{eqnarray*}
\begin{eqnarray*}
\limsup\limits_{n\rightarrow\infty} \iint_{\mathbb{R}^{2 N}} \frac{\left||u_n(x)|-|u_n(y)|\right|^{p}}{|x-y|^{N+ps}} \mathrm{d}x\mathrm{d}y = \int_{\mathbb{R}^N}d\omega + \omega_\infty,
\end{eqnarray*}
\begin{eqnarray*}
\limsup\limits\limits_{n\rightarrow\infty}\displaystyle\int_{\mathbb{R}^N}|u_n|^{p_s^\ast}dx
= \displaystyle\int_{\mathbb{R}^N}d\zeta + \zeta_\infty.
\end{eqnarray*}
In addition, the following inequality holds
\begin{align}\label{e3.2}
S_{p, H} \nu_{\infty}^{\frac{p}{2p^{*}_{\mu,s}}}\leq \omega_\infty, \quad
\nu_{\infty} \leq C(N, \mu) \zeta_{\infty}^{\frac{2 N-\mu}{N}}.
\end{align}
\end{lemma}

In order to prove the main results, we define the energy functional of
problem \eqref{e1.1} as follows
\begin{align}\label{e3.3}
J_{\lambda}(u) :=\frac{1}{p} \widetilde{M}\left([u]_{s,
A}^{p}\right)+\frac{1}{p}
\int_{\mathbb{R}^{N}}V(x)|u|^{p}\mathrm{~d}x -\frac{\lambda}{2 p_{s,
\mu}^{*}}\iint_{\mathbb{R}^{2 N}} \frac{|u(y)|^{p_{s,
\mu}^{*}}|u(x)|^{p_{s, \mu}^{*}}}{|x-y|^{\mu}}
\mathrm{d}y\mathrm{d}x-\frac{k}{q}\int_{\mathbb{R}^{N}}|u|^{q}\mathrm{d}x.
\end{align}
Under hypothetical conditions $(V)$ and $(M)$, a simple test as in
Willem \cite{w1}, yields that $J_{\lambda}\in C^{1}(E, \mathbb{R})$
and its critical points are weak solutions of problem \eqref{e1.1},
if
\begin{align}
&\nonumber M\left([u]_{s, A}^{p}\right) \operatorname{Re}L(u,
v)+\operatorname{Re}
\int_{\mathbb{R}^{N}} V(x)|u|^{p-2} u \bar{v} \mathrm{~d}x \\
&=\operatorname{Re}
\int_{\mathbb{R}^{N}}\left[\lambda\left(\int_{\mathbb{R}^{N}}
\frac{|u|^{p_{\mu, s}^{*}}}{|x-y|^{\mu}} \mathrm{d}
y\right)|u|^{p_{\mu, s}^{*}-2} u+k|u|^{q-2} u\right]\bar{v}
\mathrm{~d} x,
\end{align}
where
\begin{eqnarray}
L(u,v) = \iint_{\mathbb{R}^{2N}}\frac{|u(x)-e^{i(x-y)\cdot
A(\frac{x+y}{p})}u(y)|^{p-2}(u(x)-e^{i(x-y)\cdot
A(\frac{x+y}{p})}u(y))\overline{(v(x)-e^{i(x-y)\cdot
A(\frac{x+y}{p})}v(y))}}{|x-y|^{N+ps}}dx dy
\end{eqnarray}
and $v\in E.$

Next, we state and prove the following lemma.
\begin{lemma}\label{lem3.3}
Assume that conditions $(V)$ and $(M)$ hold. Then any $(PS)_{c}$ sequence
$\{u_{n}\}_{n}$ for $J_{\lambda}$ is bounded in $E$ and $c\geq0$.
\end{lemma}
\begin{proof}
Suppose that  $\{u_{n}\}_{n} \subset E$ is a  $(PS)$ sequence for
$J_{\lambda}.$ Then  we have
\begin{align}\label{e3.6}
c + o_n(1) = J_{\lambda}(u_{n})= & \nonumber \frac{1}{p}
\widetilde{M}\left([u_{n}]_{s,
A}^{p}\right)+\frac{1}{p}\int_{\mathbb{R}^{N}}V(x)|u_{n}|^{p}\mathrm{~d}x
\\
&-\frac{\lambda}{2 p_{s, \mu}^{*}} \iint_{\mathbb{R}^{2 N}}
\frac{|u_{n}(y)|^{p_{s, \mu}^{*}}|u_{n}(x)|^{p_{s,
\mu}^{*}}}{|x-y|^{\mu}} \mathrm{d} y \mathrm{~d} x-\frac{k}{q}\int_{\mathbb{R}^{N}}|u_n|^{q}\mathrm{d} x
\end{align}
and
\begin{align}\label{e3.7}
\left\langle J_{\lambda}^{\prime}\left(u_{n}\right),
v\right\rangle=& \nonumber
\operatorname{Re}\left\{M\left(\left[u_{n}\right]_{s, A}^{p}\right)
L\left(u_{n}, v\right)+\int_{\mathbb{R}^{N}}
V(x)\left|u_{n}\right|^{p-2} u_{n} \bar{v} \mathrm{~d} x\right.\\
&\nonumber\left.-\lambda\int_{\mathbb{R}^{N}}\left[\left(\int_{\mathbb{R}^{N}}
\frac{|u_{n}|^{p_{\mu, s}^{*}}}{|x-y|^{\mu}}\mathrm{d}
y\right)|u_{n}|^{p_{\mu, s}^{*}-2} u_{n}(x)
\bar{v}\right]\mathrm{~d}x-k\int_{\mathbb{R}^{N}}|u_{n}|^{q-2}u_{n}\bar{v}
\mathrm{~d}x\right\}\\
=&o(1)\left\|u_{n}\right\|_{E}.
\end{align}
It follows from \eqref{e3.6}, \eqref{e3.7} and $(M)$ that
\begin{align}\label{e3.8}
c+o(1)\left\|u_{n}\right\|_{\lambda} \nonumber \geq&
J_{\lambda}(u_{n})-\frac{1}{q} \left\langle
J_{\lambda}^{\prime}\left(u_{n}\right),
u_{n}\right\rangle \\
\nonumber =& \frac{1}{p} \widetilde{M}\left([u_{n}]_{s,
A}^{p}\right)-\frac{1}{q}M\left([u_{n}]_{s, A}^{p}\right)[u_{n}]_{s,
A}^{p} + \left(\frac{1}{p}-\frac{1}{q}\right) \int_{\mathbb{R}^{N}}V(x)|u_{n}|^{p}\mathrm{~d}x\\
&\nonumber +\left(\frac{1}{q}-\frac{1}{2 p_{s,
\mu}^{*}}\right)\lambda\iint_{\mathbb{R}^{2 N}} \frac{|u_{n}(y)|^{p_{s,
\mu}^{*}}|u_{n}(x)|^{p_{s, \mu}^{*}}}{|x-y|^{\mu}} \mathrm{d} y
\mathrm{~d} x \\
\geq&\nonumber\left(\frac{\sigma}{p}-\frac{1}{q}\right)m_0[u_{n}]_{s,
A}^{p} + \left(\frac{1}{p}-\frac{1}{q}\right) \int_{\mathbb{R}^{N}}V(x)|u_{n}|^{p}\mathrm{~d}x\\
\geq&\min\left\{\left(\frac{\sigma}{p}-\frac{1}{q}\right)m_0,
\left(\frac{1}{p}-\frac{1}{q}\right) \right\}\|u_n\|^p_{E}.
\end{align}
This fact implies that $\{u_{n}\}_{n}$ is bounded in $E$. We also
obtain  $c\geq0$ from \eqref{e3.8}.
\end{proof}

Now, we can show that the  following compactness condition holds.
\begin{lemma}\label{lem3.4}
Assume that conditions $(V)$ and $(M)$  hold. Then  $J_{\lambda}(u)$ satisfies
 $(PS)_{c}$ condition, for all $sp<N<sp^{2}$ and
$$ c<    \left(\frac{1}{q}-\frac{1}{2 p_{s,
\mu}^{*}}\right)\lambda^{-\frac{p}{2p_{s,
\mu}^{*}-p}}\left(m_0S_{p,H}\right)^{\frac{2p_{s,
\mu}^{*}}{2p_{s,
\mu}^{*}-p}}.$$
\end{lemma}
\begin{proof}
Let $\{u_{n}\}_{n}$ be a $(PS)_{c}$ sequence for $J_{\lambda}$.
Then, by Lemma \ref{lem3.3}, we know that the sequence
$\{u_{n}\}_{n}$ is bounded in $E$. Moreover, we know that there
exists a subsequence, still denoted by $\{u_{n}\}$, such that
$u_{n}\rightharpoonup u$ weakly in $E$. Moreover, we have
\begin{equation}\label{e3.9}
u_{n} \rightarrow u \ \ \text { a.e. in } \mathbb{R}^{N}, \quad
u_{n} \rightarrow u \ \ \text { in }
L^{s}\left(\mathbb{R}^{N}\right), 1 \leq s<p_{s}^{*}.
\end{equation}
Now, by the concentration-compactness principle, we may assume that
there exist bounded non-negative measures $\omega, \zeta$ and $\nu$
on $\mathbb{R}^{N}$, and an at most countable set
$\left\{x_{i}\right\}_{i \in I} \in \Omega \backslash\{0\}$ such
that
$$
\int_{\mathbb{R}^{N}}
\frac{\left||u_{n}(x)|-|u_{n}(y)|\right|^{p}}{|x-y|^{N+p s}}
\mathrm{~d} x \rightarrow \omega,
\quad\left|u_{n}\right|^{p_{s}^{*}} \rightarrow \zeta
$$
and
$$
\left(\int_{\mathbb{R}^{N}} \frac{\left|u_{n}\right|^{p_{\mu, s}^{*}}}{|x-y|^{\mu}} \mathrm{d} y\right)\left|u_{n}\right|^{p_{\mu, s}^{*}} \rightarrow \nu.
$$
Now, there exists a countable sequence of points $\{x_{i}\} \subset
\mathbb{R}^{N}$ and families of positive numbers $\{\nu_i : i \in
I\}$, $\{\zeta_i : i \in I\}$ and $\{\omega_i : i \in I\}$ such that
\begin{align*}
&\omega \geq \int_{\mathbb{R}^{N}} \frac{||u(x)|-|u(y)||^{p}}{|x-y|^{N+p s}} \mathrm{~d} x+\sum_{i \in I} \omega_{i}\delta_{x_{i}},\\
&\zeta = |u|^{p_{s}^{*}}+\sum_{i \in I} \zeta_{i} \delta_{x_{i}},\\
&\nu=\left(\int_{\mathbb{R}^{N}} \frac{|u|^{p_{\mu,
s}^{*}}}{|x-y|^{\mu}} \mathrm{d} y\right)|u|^{p_{\mu,
s}^{*}}+\sum_{i \in I} \nu_{i} \delta_{x_{i}}.
\end{align*}
We can also get
\begin{align}\label{e3.10}
S_{p, H} \nu_{i}^{\frac{p}{2p^{*}_{\mu,s}}}\leq \omega_{i}, \quad
\nu_{i} \leq C(N, \mu) \zeta_{i}^{\frac{2 N-\mu}{N}}.
\end{align}
In the sequel, we shall prove that
\begin{equation}\label{e3.11}
I = \emptyset.
\end{equation}
Suppose to the contrary, that $I \neq \emptyset$. Then we can define
a smooth cut-off function such that $\phi \in
C_0^\infty(\mathbb{R}^N)$ and $0 \leq \phi \leq 1$; $\phi \equiv 1$
in $B(x_i, \epsilon)$, $\phi(x) = 0$ in $\mathbb{R}^N \setminus
B(x_i, 2\epsilon)$.  Let $\epsilon
> 0$ and $\phi_\epsilon^i =
\phi\left(\frac{x-x_i}{\epsilon}\right)$, where $i \in I$. It is not
difficult to see that $\{u_n\phi_\epsilon^i\}_n$ is bounded in $E$.
Then $\langle J_\lambda'(u_n), u_n\phi_\epsilon^i\rangle \rightarrow
0$, which implies
\begin{align}\label{e3.12}
&\nonumber
M\left([u_n]_{s,A}^p\right)\iint_{\mathbb{R}^{2N}}\frac{|u_n(x)-e^{i(x-y)\cdot
A(\frac{x+y}{p})}u_n(y)|^p\phi_\epsilon^i(y)}{|x-y|^{N+ps}}dx dy + \int_{\mathbb{R}^N}V(x)|u_n|^p\phi_\epsilon^i(x) dx   \\
&\nonumber\mbox{}= - \mbox{Re}\left\{
M\left([u_n]_{s,A}^p\right)\mathcal{L}\left(u_{n}, u_n\phi_\epsilon^i\right)\right\} + \lambda \iint_{\mathbb{R}^{2 N}}
\frac{|u_{n}(y)|^{p_{s, \mu}^{*}}|u_{n}(x)|^{p_{s,
\mu}^{*}}\phi_\epsilon^i(x)}{|x-y|^{\mu}} \mathrm{d}y \mathrm{~d}x \\
&\quad +
k\int_{\mathbb{R}^{N}}|u_{n}|^{q}\phi_\epsilon^i(x)
\mathrm{~d}x +o_n(1),
\end{align}
where
\begin{align*}
&\mathcal{L}\left(u_{n}, u_n\phi_\epsilon^i\right)\\
& = \iint_{\mathbb{R}^{2N}}\frac{|u_n(x)-e^{i(x-y)\cdot
A(\frac{x+y}{p})}u_n(y)|^{p-2}(u_n(x)-e^{i(x-y)\cdot
A(\frac{x+y}{p})}u_n(y))\overline{u_n(x)(\phi_\epsilon^i(x)-\phi_\epsilon^i(y))}}{|x-y|^{N+ps}}dx
dy.
\end{align*}
By the  H\"{o}lder inequality, we know that
\begin{eqnarray}\label{e3.13}
&&\nonumber\left|\mbox{Re}\left\{
M\left([u_n]_{s,A}^p\right)\mathcal{L}\left(u_{n}, u_n\phi_\epsilon^i\right)\right\}\right|\\
&& \nonumber\mbox{} \  \leq
C\left(\iint_{\mathbb{R}^{2N}}\frac{|u_n(x)-e^{i(x-y)\cdot
A(\frac{x+y}{p})}u_n(y)|^{p}}{|x-y|^{N+ps}}dxdy\right)^{(p-1)/p}
\left(\iint_{\mathbb{R}^{2N}}\frac{|u_n(x)|^p|\phi_\epsilon^i(x)-\phi_\epsilon^i(y)|^p}{|x-y|^{N+ps}}dxdy\right)^{1/p}
\\
&& \mbox{} \  \leq C
\left(\iint_{\mathbb{R}^{2N}}\frac{|u_n(x)|^p|\phi_\epsilon^i(x)-\phi_\epsilon^i(y)|^p}{|x-y|^{N+ps}}dxdy\right)^{1/p}.
\end{eqnarray}
On the other hand, as in the proof of Lemma 3.4 in Zhang et al.
\cite{zhang2}, we can get that
\begin{eqnarray}\label{e3.14}
\lim_{\epsilon\rightarrow
 0}\lim_{n\rightarrow\infty}\iint_{\mathbb{R}^{2N}}\frac{|u_n(x)|^p|\phi_\epsilon^i(x)-\phi_\epsilon^i(y)|^p}{|x-y|^{N+ps}}dxdy
=0.
\end{eqnarray}
It follows from  \eqref{e3.12}--\eqref{e3.14} and the diamagnetic inequality that
\begin{eqnarray}\label{e3.15}
m_0\omega_i \leq \lambda\nu_i.
\end{eqnarray}
This fact together with \eqref{e3.8} implies that
$${\rm (I)} \quad \nu_i = 0 \ \indent \mbox{or}\ \quad {\rm (II)} \quad  \nu_i \geq \left(\lambda^{-1}m_0S_{p,H}\right)^{\frac{2p_\mu^*}{2p_\mu^*-p}}.$$
If (II) occurs for some $i_0 \in I$,  then
\begin{align}\label{12'}
c =& \nonumber\lim_{n\rightarrow\infty}\left(
J_{\lambda}(u_{n})-\frac{1}{q} \left\langle
J_{\lambda}^{\prime}\left(u_{n}\right),
u_{n}\right\rangle \right)\\
\nonumber \geq&  \left(\frac{1}{q}-\frac{1}{2 p_{s,
\mu}^{*}}\right)\lambda\iint_{\mathbb{R}^{2 N}} \frac{|u_{n}(y)|^{p_{s,
\mu}^{*}}|u_{n}(x)|^{p_{s, \mu}^{*}}}{|x-y|^{\mu}} \mathrm{d} y
\mathrm{~d} x \\
\geq&\left(\frac{1}{q}-\frac{1}{2 p_{s,
\mu}^{*}}\right)\lambda\nu_i \geq  \left(\frac{1}{q}-\frac{1}{2 p_{s,
\mu}^{*}}\right)\lambda^{-\frac{p}{2p_{s,
\mu}^{*}-p}}\left(m_0S_{p,H}\right)^{\frac{2p_{s,
\mu}^{*}}{2p_{s,
\mu}^{*}-p}}.
\end{align}
This is an obvious contradiction to the choice of $c$. This
completes the proof of \eqref{e3.11}.

Next, we shall prove the  concentration at infinity. To this end,
set $\phi_{R}\in C_{0}^{\infty}(\mathbb{R}^{N})$ for $R>0$, and
satisfies $\phi_{R}(x)=0$ for $|x|<R$, $\phi_{R}(x)=1$ for $|x|>2R$,
$0\leq \phi_{R} \leq 1$, and $|\nabla\phi_{R}|\leq {\frac{2}{R}}$.
Invoking Theorem $2.4$ of Xiang and Zhang \cite{a3}, we define
\begin{eqnarray*}
\omega_\infty = \lim_{R\rightarrow\infty}\limsup_{n\rightarrow\infty}\iint_{\mathbb{R}^{2 N}} \frac{\left||u_n(x)|-|u_n(y)|\right|^{p}\phi_{R}(x)}{|x-y|^{N+ps}} \mathrm{d}x\mathrm{d}y,
\end{eqnarray*}
\begin{eqnarray*}
\zeta_\infty = \lim_{R\rightarrow\infty}\limsup_{n\rightarrow\infty}\int_{\mathbb{R}^{N}}|u_{n}|^{p_{s}^{*}}\phi_{R} \mathrm{d}x
\end{eqnarray*} and
\begin{eqnarray*}
\nu_\infty = \lim_{R\rightarrow\infty}\limsup_{n\rightarrow\infty}\iint_{\mathbb{R}^{2N}} \frac{|u_{n}(x)|^{p_{\mu, s}^{*}}|u_{n}(y)|^{p_{\mu, s}^{*}}\phi_{R}(x)}{|x-y|^{\mu}} \mathrm{d} x \mathrm{d}y.
\end{eqnarray*}
By Lemma \ref{lem3.2}, we have
\begin{eqnarray*}
\limsup\limits\limits_{n\rightarrow\infty}\iint_{\mathbb{R}^{2N}} \frac{|u_{n}(x)|^{p_{\mu, s}^{*}}|u_{n}(y)|^{p_{\mu, s}^{*}}\phi_{R}(x)}{|x-y|^{\mu}} \mathrm{d} x \mathrm{d}y
= \displaystyle\int_{\mathbb{R}^N}d\nu + \nu_\infty,
\end{eqnarray*}
\begin{eqnarray*}
\limsup\limits_{n\rightarrow\infty} \iint_{\mathbb{R}^{2 N}} \frac{\left||u_n(x)|-|u_n(y)|\right|^{p}\phi_{R}(x)}{|x-y|^{N+ps}} \mathrm{d}x\mathrm{d}y = \int_{\mathbb{R}^N}d\omega + \omega_\infty,
\end{eqnarray*}
\begin{eqnarray*}
\limsup\limits\limits_{n\rightarrow\infty}\displaystyle\int_{\mathbb{R}^N}|u_n|^{p_s^\ast}dx
= \displaystyle\int_{\mathbb{R}^N}d\zeta + \zeta_\infty.
\end{eqnarray*}
Moreover,
\begin{align}\label{e3.81}
S_{p, H} \nu_{\infty}^{\frac{p}{2p^{*}_{\mu,s}}}\leq \omega_\infty, \quad
\nu_{\infty} \leq C(N, \mu) \zeta_{\infty}^{\frac{2 N-\mu}{N}}.
\end{align}
Similar discussion as above  yields
$${\rm (III)} \quad \nu_\infty = 0 \ \indent \mbox{or}\ \quad {\rm (IV)} \quad  \nu_\infty \geq \left(\lambda^{-1}m_0S_{p,H}\right)^{\frac{2p_\mu^*}{2p_\mu^*-p}}.$$
Furthermore, proceeding as in the proof of \eqref{e3.14}, we  can get $\nu_\infty = 0$.
Thus
\begin{align}\label{e3.15}
\iint_{\mathbb{R}^{2 N}} \frac{|u_{n}(y)|^{p_{s,
\mu}^{*}}|u_{n}(x)|^{p_{s, \mu}^{*}}}{|x-y|^{\mu}} \mathrm{d}y
\mathrm{~d}x \rightarrow \iint_{\mathbb{R}^{2 N}} \frac{|u(y)|^{p_{s,
\mu}^{*}}|u(x)|^{p_{s, \mu}^{*}}}{|x-y|^{\mu}} \mathrm{d}y\mathrm{~d}x \quad\mbox{as}\ n \rightarrow \infty.
\end{align}
By the Br\'{e}zis-Lieb Lemma \cite{Br}, we have
\begin{align}\label{e3.16}
&\nonumber\iint_{\mathbb{R}^{2 N}} \frac{|u_{n}(x)|^{p_{\mu, s}^{*}}|u_{n}(y)|^{p_{\mu, s}^{*}}}{|x-y|^{\mu}} \mathrm{d} x \mathrm{~d} y =\iint_{\mathbb{R}^{2 N}} \frac{|u(x)|^{p_{\mu, s}^{*}}|u(y)|^{p_{\mu, s}^{*}}}{|x-y|^{\mu}} \mathrm{d} x \mathrm{~d} y \\
&\quad +\iint_{\mathbb{R}^{2 N}} \frac{|u_{n}(x)-u(x)|^{p_{\mu, s}^{*}}|u_{n}(y)-u(y)|^{p_{\mu, s}^{*}}}{|x-y|^{\mu}} \mathrm{d} x \mathrm{d} y+o(1).
\end{align}
Hence, \eqref{e3.15} and \eqref{e3.16} imply that
\begin{align}\label{e3.17}
\iint_{\mathbb{R}^{2N}} \frac{|u_{n}(x)-u(x)|^{p_{\mu, s}^{*}}|u_{n}(y)-u(y)|^{p_{\mu, s}^{*}}}{|x-y|^{\mu}} \mathrm{d}x \mathrm{d}y+o(1)\rightarrow0\quad\mbox{as}\quad n\rightarrow\infty.
\end{align}
Moreover, it is easy to see that
\begin{align}\label{e3.18}
\int_{\mathbb{R}^{N}}(|u_{n}(x)|^{q-2}u_{n}(x)-|u(x)|^{q-2} u(x))(u_{n}(x)-u(x))\mathrm{d} x \rightarrow0\quad\mbox{as}\quad n\rightarrow\infty.
\end{align}
By \eqref{e3.17}, \eqref{e3.18} and  the H\"{o}lder inequality, we have
\begin{align}\label{e3.19}
&\nonumber  \left\langle J_{\lambda}^{\prime}\left(u_{n}\right)- J_{\lambda}^{\prime}\left(u\right), u_{n}-u\right\rangle\\
=&\nonumber\operatorname{Re}\left\{M\left(\left[u_{n}\right]_{s, A}^{p}\right)L(u_{n},u_{n}-u)-M\left(\left[u\right]_{s, A}^{p}\right)L(u,u_{n}-u)\right.\\
&\nonumber+ \int_{\mathbb{R}^{N}} V(x)(|u_{n}(x)|^{p-2} u_{n}(x)-|u(x)|^{p-2} u(x))(u_{n}(x)-u(x))\mathrm{d} x\\
&\nonumber-\lambda\iint_{\mathbb{R}^{2 N}} \frac{|u_{n}(x)-u(x)|^{p_{\mu, s}^{*}}|u_{n}(y)-u(y)|^{p_{\mu, s}^{*}}}{|x-y|^{\mu}} \mathrm{d} x \mathrm{d} y\\
&\nonumber\left.-k\int_{\mathbb{R}^{N}} (|u_{n}(x)|^{q-2} u_{n}(x)-|u(x)|^{q-2} u(x))(u_{n}(x)-u(x))\mathrm{d} x\right\}\\
\geq&\nonumber\operatorname{Re}\left\{M\left(\left[u_{n}\right]_{s, A}^{p}\right)(\left[u_{n}\right]_{s, A}^{p})^{\frac{p-1}{p}}\left[(\left[u_{n}\right]_{s, A}^{p})^{\frac{1}{p}}-(\left[u\right]_{s, A}^{p})^{\frac{1}{p}}\right]\right.\\
&\nonumber+M\left(\left[u\right]_{s, A}^{p}\right)(\left[u\right]_{s, A}^{p})^{\frac{p-1}{p}}\left[(\left[u\right]_{s, A}^{p})^{\frac{1}{p}}-(\left[u_{n}\right]_{s, A}^{p})^{\frac{1}{p}}\right]\\
&\nonumber+\left(\int_{\mathbb{R}^{N}}V(x)|u_{n}|^{p}\mathrm{d}x\right)^{\frac{p-1}{p}}\left[\left(\int_{\mathbb{R}^{N}}V(x)|u_{n}|^{p}\mathrm{d} x\right)^{\frac{1}{p}}-\left(\int_{\mathbb{R}^{N}}V(x)|u|^{p}\mathrm{d} x\right)^{\frac{1}{p}}\right]\\
&\nonumber\left.+\left(\int_{\mathbb{R}^{N}}V(x)|u|^{p}\mathrm{d} x\right)^{\frac{p-1}{p}}\left[\left(\int_{\mathbb{R}^{N}} V(x)|u|^{p}\mathrm{d} x\right)^{\frac{1}{p}}-\left(\int_{\mathbb{R}^{N}} V(x)|u_{n}|^{p}\mathrm{~d} x\right)^{\frac{1}{p}}\right] \right\}\\
=&\nonumber\operatorname{Re}\left\{\left[\left(\left[u_{n}\right]_{s, A}^{p}\right)^{\frac{1}{p}}-(\left[u\right]_{s, A}^{p})^{\frac{1}{p}}\right][M\left(\left[u_{n}\right]_{s, A}^{p}\right)(\left[u_{n}\right]_{s, A}^{p})^{\frac{p-1}{p}}-M\left(\left[u\right]_{s, A}^{p}\right)(\left[u\right]_{s, A}^{p})^{\frac{p-1}{p}}]\right.\\
&\nonumber+\left[\left(\int_{\mathbb{R}^{N}} V(x)|u_{n}|^{p}\mathrm{~d} x\right)^{\frac{1}{p}}-\left(\int_{\mathbb{R}^{N}}V(x)|u|^{p}\mathrm{~d} x\right)^{\frac{1}{p}}\right]\\
&\times\left.\left[\left(\int_{\mathbb{R}^{N}}V(x)|u_{n}|^{p}\mathrm{d}x\right)^{\frac{p-1}{p}} -\left(\int_{\mathbb{R}^{N}}V(x)|u|^{p}\mathrm{d} x\right)^{\frac{p-1}{p}}\right]\right\}.
\end{align}
Since $u_{n}\rightharpoonup u$ in $E$ and
$J_{\lambda}'\left(u_{n}\right)\rightarrow0$ as $n\rightarrow\infty$
in $E^*$, we can conclude that
$$\left\langle J_{\lambda}'\left(u_{n}\right)- J_{\lambda}'\left(u\right), u_{n}-u\right\rangle\rightarrow0 \quad\mbox{as}\ n\rightarrow\infty.$$
It follows from $u_{n}\rightarrow u$ a.e in $\mathbb{R}^{N}$ and the
Fatou lemma  that
\begin{align}\label{e3.20}
\left[u\right]_{s, A}^{p}\leq\liminf_{n\rightarrow\infty}\left[u_n\right]_{s, A}^{p} = d_{1}
\end{align}
and
\begin{align}\label{e3.21}
\int_{\mathbb{R}^{N}} V(x)|u|^{p}\mathrm{~d} x\leq\liminf_{n\rightarrow\infty}\int_{\mathbb{R}^{N}} V(x)|u_{n}|^{p}\mathrm{~d} x=d_{2}.
\end{align}
We note that
\begin{align}\label{e3.22}
\left[(d_{1})^{\frac{1}{p}}-(\left[u\right]_{s, A}^{p})^{\frac{1}{p}}\right]\left[M(d_{1})d_{1}^{\frac{p-1}{p}}-M\left(\left[u\right]_{s, A}^{p}\right)(\left[u\right]_{s, A}^{p})^{\frac{p-1}{p}}\right]\geq0
\end{align}
and
\begin{align}\label{e3.23}
\left[(d_{2})^{\frac{1}{p}}-\int_{\mathbb{R}^{N}}V(x)|u|^{p}\mathrm{d}
x)^{\frac{1}{p}}\right]\left[(d_{2})^{\frac{p-1}{p}}-(\int_{\mathbb{R}^{N}}
V(x)|u|^{p}\mathrm{d} x)^{\frac{p-1}{p}}\right ]\geq0,
\end{align}
since $g(t)=M(t)t^{\frac{p-1}{p}}$ is nondecreasing for $t\geq0$.
Thus, by $$\left\langle J_{\lambda}'\left(u_{n}\right)-
J_{\lambda}'(u), u_{n}-u\right\rangle\rightarrow0\ \ \text{as}\ \
n\rightarrow\infty$$ and \eqref{e3.17}--\eqref{e3.23}, we get
\begin{align}\label{e3.24}
0\nonumber&\geq\liminf_{n\rightarrow\infty}\operatorname{Re}\left\{\left[\left(\left[u_{n}\right]_{s, A}^{p}\right)^{\frac{1}{p}}-(\left[u\right]_{s, A}^{p})^{\frac{1}{p}}\right][M\left(\left[u_{n}\right]_{s, A}^{p}\right)(\left[u_{n}\right]_{s, A}^{p})^{\frac{p-1}{p}}-M\left(\left[u\right]_{s, A}^{p}\right)(\left[u\right]_{s, A}^{p})^{\frac{p-1}{p}}]\right.\\
&\nonumber\quad+\left[\left(\int_{\mathbb{R}^{N}} V(x)|u_{n}|^{p}\mathrm{~d} x\right)^{\frac{1}{p}}-\left(\int_{\mathbb{R}^{N}}V(x)|u|^{p}\mathrm{~d} x\right)^{\frac{1}{p}}\right] \left.\left[\left(\int_{\mathbb{R}^{N}}V(x)|u_{n}|^{p}\mathrm{d}x\right)^{\frac{p-1}{p}} -\left(\int_{\mathbb{R}^{N}}V(x)|u|^{p}\mathrm{d} x\right)^{\frac{p-1}{p}}\right]\right\}\\
&\geq\nonumber\operatorname{Re}\left\{\left[(d_{1})^{\frac{1}{p}}-(\left[u\right]_{s,
A}^{p})^{\frac{1}{p}}\right]\left[M(d_{1})d_{1}^{\frac{p-1}{p}}-M\left(\left[u\right]_{s,
A}^{p}\right)(\left[u\right]_{s,
A}^{p})^{\frac{p-1}{p}}\right]\right.\\
&\left.\quad +\left[(d_{2})^{\frac{1}{p}}-
\left(\int_{\mathbb{R}^{N}}V(x)|u|^{p}\mathrm{d}x\right)^{\frac{1}{p}}\right]\left[(d_{2})^{\frac{p-1}{p}}-\left(\int_{\mathbb{R}^{N}}
V(x)|u|^{p}\mathrm{d}x\right)^{\frac{p-1}{p}}\right]\right\}.
\end{align}
It follows from  \eqref{e3.22}-\eqref{e3.24}  that
\begin{align*}
\iint_{\mathbb{R}^{2 N}} \frac{\left|u(x)-\mathrm{e}^{i(x-y) \cdot A(\frac{x+y}{p})}u(y)\right|^{p}}{|x-y|^{N+ps}} \mathrm{d}x\mathrm{~d} y=d_{1}
\quad\mbox{and}\quad \int_{\mathbb{R}^{N}} V(x)|u|^{p}\mathrm{~d} x=d_{2}.
\end{align*}
Then $\|u_{n}\|_{E} \rightarrow\|u\|_{E}$. We note that $E$ is a
reflexive Banach space,  thus $u_n \rightarrow u$ strongly converges
in $E$. This completes the proof of Lemma \ref{lem3.4}.
\end{proof}

\section{Auxiliary  results}\label{s4}
First, we shall prove that functional $J_{\lambda}$  has a mountain
path structure.
\begin{lemma}\label{lem4.1}
Let the conditions $(V)$ and $(M)$ hold. Then
\begin{enumerate}
\item[$(C_1)$] there exist some constants $\alpha_{\lambda},\beta_{\lambda}>0$ such that $J_{\lambda}(u)>0$
if $u\in B_{\beta_{\lambda}}\backslash\{0\}$ and
$J_{\lambda}(u)\geq\alpha_{\lambda}$ if $u\in \partial
B_{\beta_{\lambda}}$, where $B_{\beta_{\lambda}}=\{u\in
E:\|u\|_{E}\leq\beta_{\lambda}\}$;
\item[$(C_2)$] we have
\begin{displaymath}
J_\lambda(u) \rightarrow -\infty\quad \mbox{as}\quad \ u\in F
\subset E, \ \|u\|_E \rightarrow \infty,
\end{displaymath}
where $F$ is a  finite-dimensional subspace of $E$.
\end{enumerate}
\end{lemma}
\begin{proof}
It follows from the Hardy-Littlewood-Sobolev inequality that there
exists $C(N,\mu)>0$ such that
$$\iint_{\mathbb{R}^{2 N}} \frac{|u_{n}(x)|^{p_{\mu, s}^{*}}|u_{n}(y)|^{p_{\mu, s}^{*}}}{|x-y|^{\mu}} \mathrm{d} x \mathrm{~d} y\leq C(N,\mu)\|u\|^{2p_{\mu,s}^{*}}_{E} \ \ \hbox{for all} \ \ u\in E.$$

By virtue of  $(V)$ and $(M)$, we get
\begin{align}\label{e4.1}
J_{\lambda}(u)
\geq\min\left\{\frac{\sigma\alpha_{0}}{p},\frac{1}{p}\right\}
\|u\|^{p}_{E}-\frac{\lambda}{2 p_{s,
\mu}^{*}}C(N,\mu)\|u\|^{2p_{s,\mu}^{*}}_{E}-Ck\|u\|^q.
\end{align}

Since $p_{s, \mu}^{*}, q > p$, we know that the  conclusion $(C_1)$
of Lemma \ref{lem4.1} holds.

In order to prove the  conclusion $(C_2)$ of Lemma \ref{lem4.1}, we
note that it follows from the condition $(m_{2})$ that
\begin{align}\label{e4.2}
\widetilde{M}(t)\leq{\frac{\widetilde{M}(t_{0})}{t_{0}^{\frac{1}{\sigma}}}}{t^{\frac{1}{\sigma}}}=C_{0}t^{\frac{1}{\sigma}}\
\ \mbox{for all}\quad t\geq t_{0}>0.
\end{align}
Let $\omega \in C_0^\infty(\mathbb R^{N}, \mathbb{C})$ with
$\|\omega\| = 1$. Thus
\begin{displaymath}
J_\lambda(t\omega) \leq \frac{C_{0}}{p}t^{\frac{p}{\sigma}}  +
\frac{1}{p}t^{p}
- \frac{\lambda}{2 p_{s, \mu}^{*}}t^{2 p_{s, \mu}^{*}}\iint_{\mathbb{R}^{2 N}} \frac{|\omega(y)|^{p_{s, \mu}^{*}}|\omega(x)|^{p_{s, \mu}^{*}}}{|x-y|^{\mu}} \mathrm{d} y \mathrm{~d} x
- \frac{k}{q}t^q|\omega|_q^q.
\end{displaymath}
Note that  all norms are equivalent  in a finite-dimensional space.
Then the above fact  together with $p<\frac{p}{\sigma}<2 p_{s,
\mu}^{*}$ implies that the conclusion $(C_2)$ of Lemma \ref{lem4.1}
holds.
\end{proof}

Invoking  Binlin et al.  \cite[Theorem 3.2]{a9}, we have
$$ \inf\left\{\iint_{\mathbb{R}^{2 N}}\frac{|\phi(x)-\phi(y)|^{p}}{|x-y|^{N+ps}}\mathrm{~d} x\mathrm{~d} y:\phi\in C_{0}^{\infty}(\mathbb{R}^{N}),|\phi|_{q}=1\right\}=0.$$
For any $1>\zeta>0$, let $\phi_{\zeta}\in
C_{0}^{\infty}(\mathbb{R}^{N})$ with $|\phi_{\zeta}|_{q}=1$ and supp
$\phi_{\zeta}\subset B_{r_{\zeta}}(0)$ be such  that
$$\iint_{\mathbb{R}^{2 N}}\frac{|\phi_{\zeta}(x)-\phi_{\zeta}(y)|^{p}}{|x-y|^{N+ps}}\mathrm{~d} x\mathrm{~d} y\leq C\zeta^{(pN-(N-ps)q)/q}$$
and define
\begin{equation}\label{e4.3}
\psi_{\zeta}(x)=e^{iA(0)x}\phi_{\zeta}(x), \quad  \psi_{\lambda,\zeta}(x)=\psi_{\zeta}(\lambda ^{-\tau}x)
\end{equation}
and
\begin{equation}\label{e4.4}
\tau =\frac{1} {(N-ps)}\left(-\frac{p}{2p_{s,
\mu}^{*}-p} \right).
\end{equation}
So,  we have
\begin{equation*}
\begin{aligned}
J_{\lambda}(t\psi_{\lambda,\zeta})&\leq \frac{C_{0}}{p}t^{p/\sigma}\left(\iint_{\mathbb{R}^{2 N}}\frac{|\psi_{\lambda,\zeta}(x)-e^{i(x-y)\cdot A((x+y)/p)}\psi_{\lambda,\zeta}(y)|^{p}}{|x-y|^{N+ps}}\mathrm{~d}x\mathrm{~d} y\right)^{1/\sigma}\\
&\quad+\frac{t^{p}}{p}\int_{\mathbb{R}^{N}} V(x)|\psi_{\lambda,\zeta}|^{p} \mathrm{~d} x-t^{q}\frac{k}{q}\int_{\mathbb{R}^{N}}|\psi_{\lambda,\zeta}|^{q}\mathrm{d} x\\
&\leq\lambda^{\tau(N-ps)}\left[\frac{C^{0}}{p}t^{p/\sigma}\left(\iint_{\mathbb{R}^{2N}}\frac{|\psi_{\zeta}(x)-e^{i(x-y)\cdot A((\lambda^{\tau}x+\lambda^{\tau}y)/p)}\psi_{\zeta}(y)|^{p}}{|x-y|^{N+ps}}\mathrm{~d} x\mathrm{~d} y\right)^{1/\sigma}\right.\\
&\quad\left.+\frac{t^{p}}{p} \int_{\mathbb{R}^{N}}V(\lambda^{\tau}x)|\psi_{\zeta}|^{p} \mathrm{~d} x-t^{q}\frac{k}{q}\int_{\mathbb{R}^{N}}|\psi_{\zeta}|^{q}\mathrm{d} x\right]\\
&=\lambda^{-\frac{p}{2p_{s,
\mu}^{*}-p}}\Psi_{\lambda}(t\psi_{\zeta}),
\end{aligned}
\end{equation*}
where
\begin{equation*}
\begin{aligned}
\Psi_{\lambda}(u)&:=\frac{C_{0}}{p}\left(\iint_{\mathbb{R}^{2N}}\frac{|u(x)-e^{i(x-y)\cdot A((\lambda^{\tau}x+\lambda^{\tau}y)/p)}u(y)|^{p}}{|x-y|^{N+ps}}\mathrm{~d} x\mathrm{~d} y\right)^{1/\sigma}\\
&\quad+\frac{1}{p} \int_{\mathbb{R}^{N}} V(\lambda^{\tau}x)|u|^{p} \mathrm{~d} x-\frac{k}{q}\int_{\mathbb{R}^{N}}|u|^{q}\mathrm{d} x.
\end{aligned}
\end{equation*}
Since $q>p/\sigma$, we can find $t_{0}\in [0, +\infty)$ such that
\begin{equation*}
\begin{aligned}
\max_{t\geq0}\Psi_{\lambda}(t\psi_{\zeta})&\leq\frac{C^{0}}{p}t^{p/\sigma}_{0}\left(\iint_{\mathbb{R}^{2 N}}\frac{|\psi_{\zeta}(x)-e^{i(x-y)\cdot A((\lambda^{\tau}x+\lambda^{\tau}y)/p)}\psi_{\zeta}(y)|^{p}}{|x-y|^{N+ps}}\mathrm{~d} x\mathrm{~d} y\right)^{1/\sigma}\\
&\quad+\frac{t^{p}_{0}}{p} \int_{\mathbb{R}^{N}} V(\lambda^{\tau}x)|\psi_{\zeta}|^{p} \mathrm{~d}x.
\end{aligned}
\end{equation*}
Using the above analysis, we can prove the following conclusions.
\begin{lemma}\label{lem4.2}
For each $\zeta>0$, there exists $\lambda_{0}=\lambda_{0}(\zeta)>0$
such that
$$\iint_{\mathbb{R}^{2 N}}\frac{|\psi_{\zeta}(x)-e^{i(x-y)\cdot A((\lambda^{\tau}x+\lambda^{\tau}y)/p)}\psi_{\zeta}(y)|^{p}}{|x-y|^{N+ps}}\mathrm{~d} x\mathrm{~d} y\leq C\zeta^{(pN-(N-ps)q)/q}+\frac{2^{p-1}}{p-ps}\zeta^{ps}+\frac{2^{2p-1}}{ps}\zeta^{ps},$$
for all $0 < \lambda_{0} <\lambda$ and some constant $C>0$ depending
only on $[\phi]_{s,0}$.
\end{lemma}
\begin{proof}
For each $\zeta>0$, we know that
\begin{equation*}
\begin{aligned}
&\iint_{\mathbb{R}^{2 N}}\frac{|\psi_{\zeta}(x)-e^{i(x-y)\cdot A((\lambda^{\tau}x+\lambda^{\tau}y)/p)}\psi_{\zeta}(y)|^{p}}{|x-y|^{N+ps}}\mathrm{d}x\mathrm{d}y\\
&\quad=\iint_{\mathbb{R}^{2 N}}\frac{|e^{iA(0)\cdot x}\phi_{\zeta}(x)-e^{i(x-y)\cdot A((\lambda^{\tau}x+\lambda^{\tau}y)/p)}e^{iA(0)\cdot y}\phi_{\zeta}(y)|^{p}}{|x-y|^{N+ps}}\mathrm{d} x\mathrm{d}y\\
&\quad\leq2^{p-1}\iint_{\mathbb{R}^{2N}}\frac{|\phi_{\zeta}(x)-\phi_{\zeta}(y)|^{p}}{|x-y|^{N+ps}}\mathrm{d}x\mathrm{d}y+2^{p-1}\iint_{\mathbb{R}^{2
N}}\frac{|\phi_{\zeta}(y)|^{p}|e^{i(x-y)\cdot(A(0)-A((\lambda^{\tau}x+\lambda^{\tau}y)/p))}-1|^{p}}{|x-y|^{N+ps}}\mathrm{d}x\mathrm{d}y.
\end{aligned}
\end{equation*}
Note that
\begin{equation}
\begin{aligned}\label{22}
\left|e^{i(x-y)\cdot(A(0)-A((\lambda^{\tau}x+\lambda^{\tau}y)/p))}-1\right|^{p}=2^{p}\sin^{p}\left[\frac{(x-y)\cdot
(A(0)-A(\frac{\lambda^{\tau}x+\lambda^{\tau} y}{p}))}{p}\right].
\end{aligned}
\end{equation}
Let $y\in B_{r_{\zeta}}$ and take  $|x-y|\leq
1/\zeta|\phi_{\zeta}|_{L^{p}}^{1/s}$ such that $|x|\leq
r_{\zeta}+1/\zeta|\phi_{\zeta}|_{L^{p}}^{1/s}.$ Then, we have
$$\left|\frac{\lambda^{\tau}x+\lambda^{\tau}y}{p}\right|\leq\frac{\lambda^{\tau}}{p}\left(2r_{\zeta}+\frac{1}{\zeta}|\phi_{\zeta}|_{L^{p}}^{1/s}\right).$$
By the continuity of the function $A$,   there exists
$\lambda_{0}>0$ such that for any $\lambda > \lambda_{0}$, one has
$$\left|A(0)-A\left(\frac{\lambda^{\tau}x+\lambda^{\tau}y}{p}\right)\right|\leq \zeta|\phi_{\zeta}|_{L^{p}}^{-1/s}\quad\mbox{for}\quad |y|\leq r_{\zeta}\quad\mbox{and}\quad|x|\leq r_{\zeta}+\frac{1}{\zeta}|\phi_{\zeta}|_{L^{p}}^{1/s}$$
which means
$$\left|e^{i(x-y)\cdot(A(0)-A((\lambda^{\tau}x+\lambda^{\tau}y)/p))}-1\right|^{p} \leq |x-y|^{p}\zeta^{p}|\phi_{\zeta}|^{-p/s}_{L^{p}}.$$
Let $\zeta>0$ and $y\in B_{r_{\zeta}}$, and define
$$N_{\zeta,y}:= \left\{x\in \mathbb{R}^{N}:|x-y|\leq \frac{1}{\zeta}|\phi_{\zeta}|_{L^{p}}^{1/s}\right\}.$$
Then, for all $\lambda > \lambda_{0} > 0$, we get
\begin{equation*}
\begin{aligned}
&\iint_{\mathbb{R}^{2 N}}\frac{|\phi_{\zeta}(y)|^{p}|e^{i(x-y)\cdot(A(0)-A((\lambda^{\tau} x+\lambda^{\tau} y)/p)}-1|^{p}}{|x-y|^{N+ps}}\mathrm{~d} x\mathrm{~d} y\\
&=\int_{B_{r_{\zeta}}}|\phi_{\zeta}(y)|^{p}\mathrm{~d} y\int_{N_{\zeta,y}}\frac{|e^{i(x-y)\cdot(A(0)-A((\lambda^{\tau}x+\lambda^{\tau} y)/p)}-1|^{p}}{|x-y|^{N+ps}}\mathrm{~d} x\\
&\quad+\int_{B_{r_{\zeta}}}|\phi_{\zeta}(y)|^{p}\mathrm{~d} y\int_{\mathbb{R}^{N}\backslash N_{\zeta,y}}\frac{|e^{i(x-y)\cdot(A(0)-A((\lambda^{\tau} x+\lambda^{\tau} y)/p)}-1|^{p}}{|x-y|^{N+ps}}\mathrm{~d} x\\
&\leq \int_{B_{r_{\zeta}}}|\phi_{\zeta}(y)|^{p}\mathrm{~d} y\int_{N_{\zeta,y}}\frac{|x-y|^{p}}{|x-y|^{N+ps}}\zeta^{p}|\phi_{\zeta}|_{L^{p}}^{-p/s}\mathrm{~d} x\\
&\quad+\int_{B_{r_{\zeta}}}|\phi_{\zeta}(y)|^{p}\mathrm{~d} y\int_{\mathbb{R}^{N}\backslash N_{\zeta,y}}\frac{2p}{|x-y|^{N+ps}}\mathrm{~d} x\\
&\leq \frac{1}{p-ps}\zeta^{ps}+\frac{2^{p}}{ps}\zeta^{ps}.
\end{aligned}
\end{equation*}
This completes the proof of Lemma \ref{lem4.2}.
\end{proof}

It follows from $V(0)=0$ and supp$\phi_{\varsigma}\subset
B_{r_{\varsigma}}(0)$ that
$$V(\lambda^{\tau}x)\leq \frac{\zeta}{|\phi_{\zeta}|^{p}_{p}}\quad\mbox{for all}\quad |x|\leq r_{\zeta}\quad\mbox{and}\quad \lambda > \lambda^{*}.$$
Thus
\begin{equation}
\begin{aligned}\label{23}
\max_{t\geq0}\Psi_{\lambda}(t\phi_{\delta})\leq \frac{C_{0}}{p}
t_{0}^{p/\sigma}(C\zeta^{(pN-(N-ps)q)/q}+\frac
{2^{p-1}}{p-ps}\zeta^{ps}+\frac{2^{2p-1}}{ps}\zeta^{ps})^{1/\sigma}+\frac{t_{0}^{p}}{p}\zeta,
\end{aligned}
\end{equation}
where $C > 0$ and $C_{0} > 0$.  So, for any $\lambda >
\max\{\lambda_{0}, \lambda^{*}\}$, we can get
\begin{equation}
\begin{aligned}\label{24}
\max_{t\geq0}J_{\lambda}(t\psi_{\lambda,\zeta}) \leq \left[\frac
{C_{0}}{p}t_{0}^{p/\sigma}\left(C\zeta^{(pN-(N-ps)q)/q}+\frac
{2^{p-1}}{p-ps}\zeta^{ps}+\frac{2^{p-1}}{ps}\zeta^{ps}\right)^{1/\sigma}+\frac{t_{0}^{p}}{p}\zeta\right]\lambda^{-\frac{p}{2p_{s,
\mu}^{*}-p}}.
\end{aligned}
\end{equation}
So we have the following conclusion.
\begin{lemma}\label{lem4.3}
Let the conditions $(V)$ and $(M)$ hold. Then  for each $\kappa>0$,
there exists $\lambda_{\kappa}>0$ such that for any
$0<\lambda_{\kappa}<\lambda$, and $\widetilde{e_{\lambda}}\in E$ we
have that
$\|\widetilde{e_{\lambda}}\|>\varrho_{\lambda},J_{\lambda}(t\widetilde{e_{\lambda}})\leq0$
and
\begin{equation}
\begin{aligned}\label{25}
\max_{t\in [0,1]}J_{\lambda}(t\widetilde{e_{\lambda}})\leq \kappa\lambda^{-\frac{p}{2p_{s,
\mu}^{*}-p}}.
\end{aligned}
\end{equation}
\end{lemma}
\begin{proof}
Select $\zeta>0$ so small that
$$\frac {C_{0}}{p}t_{0}^{p/\sigma}\left(C\zeta^{(pN-(N-ps)q)/q}+\frac {2^{p-1}}{p-ps}\zeta^{ps}+\frac{2^{p-1}}{ps}\zeta^{ps}\right)^{1/\sigma}+\frac{t_{0}^{p}}{p}\zeta\leq \kappa.$$
Let $\psi_{\lambda,\zeta}\in E$ be the function defined by
\eqref{e4.3}. Let
$\lambda_{\kappa}=\min\{\lambda_{0},\lambda^{*}\}$, choose
$\widetilde{t_{\lambda}}>0$ such that
$\widetilde{t_{\lambda}}\|\psi_{\lambda,\zeta}\|>\varrho_{\lambda}$
and $J_{\lambda}(t\psi_{\lambda,\zeta})\leq0$ for all
$t\geq\widetilde{t_{\lambda}}$. By \eqref{24},  setting
$\widetilde{e_{\lambda}}=
\widetilde{t_{\lambda}}\psi_{\lambda,\zeta}$ we can obtain the
conclusion of Lemma \ref{lem4.3}.
\end{proof}

Now, fix $m^{*}\in N$. Then we can  select $m^{*}$ functions
$\phi_{\zeta}^{i}\in C_{0}^{\infty}(\mathbb{R}^{N})$ such that supp
$\phi_{\zeta}^{i}\cap$ supp $\phi_{\zeta}^{k}=\emptyset,i\neq
k,|\phi_{\zeta}^{i}|_{s}=1$ and
$$\iint_{\mathbb{R}^{2 N}}\frac{|\phi_{\zeta}^{i}(x)-\phi_{\zeta}^{i}(y)|^{p}}{|x-y|^{N+ps}}\mathrm{~d} x\mathrm{~d} y\leq C\zeta^{(pN-(N-ps)q)/q}.$$
Let $r_{\zeta}^{m^{*}}>0$ be such that supp $\phi_{\zeta}^{i}\subset
B_{r\zeta}^{i}(0)$ for $i=1,2,\cdots,m^{*}.$  Define
\begin{equation}
\begin{aligned}\label{26}
\psi_{\zeta}^{i}(x)=e^{iA(0)x}\phi_{\zeta}^{i}(x)
\end{aligned}
\end{equation}
and
\begin{equation}
\begin{aligned}\label{27}
\psi_{\lambda,\zeta}^{i}(x)=\psi_{\zeta}^{i}(\lambda^{-\tau} x).
\end{aligned}
\end{equation}
Let
$$H_{\lambda\zeta}^{m^{*}}=\text{span}\{\psi_{\lambda,\zeta}^{1},\psi_{\lambda,\zeta}^{2},\cdots,\psi_{\lambda,\zeta}^{m^{*}}\}.$$
Since for each
$u=\Sigma_{i=1}^{m^{*}}c_{i}\psi^{i}_{\lambda,\zeta}\in
H_{\lambda\zeta}^{m^{*}}$, we have
$$[u]_{s,A}^{p}\leq C\sum_{i=1}^{m^{*}}|c_{i}|^{p}[\psi^{i}_{\lambda,\zeta}]^{p}_{s,A},$$
$$ \int_{\mathbb{R}^{N}} V(x)|u|^{p} \mathrm{~d} x=\sum_{i=1}^{m^{*}}|c_{i}|^{p}\int_{\mathbb{R}^{N}} V(x)|\psi^{i}_{\lambda,\zeta}|^{p} \mathrm{~d} x$$
and
\begin{equation*}
\begin{aligned}
&\frac{1}{2 p_{s, \mu}^{*}} \iint_{\mathbb{R}^{2 N}} \frac{|u(y)|^{p_{s, \mu}^{*}}|u(x)|^{p_{s, \mu}^{*}}}{|x-y|^{\mu}} \mathrm{d} y \mathrm{~d} x+\frac{1}{q} \int_{\mathbb{R}^{N}}|u|^{q}\mathrm{d}x\\
&=\sum_{i=1}^{m^{*}}(\frac{1}{2 p_{s, \mu}^{*}}\iint_{\mathbb{R}^{2
N}} \frac{|c_{i}\psi_{\lambda,\zeta}^{i}(y)|^{p_{s,
\mu}^{*}}|c_{i}\psi_{\lambda,\zeta}^{i}(x)|^{p_{s,
\mu}^{*}}}{|x-y|^{\mu}} \mathrm{d} y \mathrm{~d} x+\frac{1}{q}
\int_{\mathbb{R}^{N}}|c_{i}\psi_{\lambda,\zeta}^{i}|^{q})\mathrm{d}x.
\end{aligned}
\end{equation*}
Hence
$$J_{\lambda}(u)\leq C\sum_{i=1}^{m^{*}}J_{\lambda}(c_{i}\psi_{\lambda,\zeta}^{i}).$$
for $C>0$. Similar to the previous discussion, we have
$$J_{\lambda}(c_{i}\psi_{\lambda,\zeta}^{i})\leq\lambda^{-\frac{p}{2p_{s,
\mu}^{*}-p}}\Psi(|c_{i}|\psi_{\zeta}^{i}),$$ and we can get the
following estimate
\begin{equation}\label{28}
 \max_{u\in H_{\lambda\delta}^{m^{*}}}J_{\lambda}(u) \leq
Cm^{*}\left[\frac{C_{0}}{p}t_{0}^{p/\sigma}(C\zeta^{(pN-(N-ps)q)/q}+\frac{2^{p-1}}{p-ps}\zeta^{ps}+\frac{2^{2p-1}}{ps}\zeta^{ps})^{1/\sigma}+\frac{t_{0}^{p}}{p}\zeta\right]
\lambda^{-\frac{p}{2p_{s, \mu}^{*}-p}},
\end{equation}
for any $\zeta\rightarrow 0$ and $C>0$. From \eqref{28}, we get the
following lemma.
\begin{lemma}\label{lem4.5}
Let the conditions $(V)$ and $(M)$ hold. Then for each
$m^{*}\in\mathbb{N}$, there exists $\lambda_{m^{*}}>0$ such that for
each $0<\lambda_{m^{*}}<\lambda$ and $m^{*}$-dimensional subspace
$F_{\lambda_{ m^{*}}}$ the following holds
$$\max _{u\in F_{\lambda_{ m^{*}}}}J_{\lambda}(u)\leq\kappa\lambda^{-\frac{p}{2p_{s,
\mu}^{*}-p}}.$$
\end{lemma}
\begin{proof}
Let $\zeta>0$   be small enough so that
$$Cm^{*}[\frac{C_{0}}{p}t_{0}^{p/\sigma}(C\zeta^{(pN-(N-ps)q)/q}+\frac{2^{p-1}}{p-ps}\zeta^{ps}+\frac{2^{2p-1}}{ps}\zeta^{ps})^{1/\sigma}+\frac{t_{0}^{p}}{p}\zeta]
\leq\kappa.$$ Set
$F_{\lambda,m^{*}}=H_{\lambda\zeta}^{m^{*}}=\hbox{span}\{\psi_{\lambda,\zeta}^{1},\psi_{\lambda,\zeta}^{2},\cdots,\psi_{\lambda,\zeta}^{m^{*}}\}.$
Thus the conclusion of Lemma \ref{lem4.5} follows from \eqref{28}.
\end{proof}

\section{Proofs of main results}\label{s5}
In the section, we shall prove the existence and multiplicity of
solutions for problem \eqref{e1.1}.
\begin{proof}[\bf Proof of Theorem 1.1]
Let $0<\kappa<\sigma_{0}$. By Lemma \ref{lem3.4}, we can select
$\lambda_{k}>0$ and for $0<\lambda<\lambda_{k}$, define the
minimax value as follows
$$c_{\lambda}:=\inf_{\gamma\in \Gamma_{\lambda}}\max_{t\in [0,1]}J_{\lambda}(\hat{te_{\lambda}}),$$
where
$$\Gamma_{\lambda}:=\{\gamma\in C([0,1],E):\gamma(0)=0\quad\mbox{and}\quad \gamma(1)=\hat{e_{\lambda}}\}.$$
By Lemma \ref{lem4.1}, we know that $$\alpha_{\lambda}\leq
c_{\lambda}\leq\kappa\lambda^{-\frac{p}{2p_{s, \mu}^{*}-p}}.$$ By
virtue of Lemma \ref{lem3.4}, we can see that $J_{\lambda}$
satisfies the $(PS)_{c}$ condition, there is $u_{\lambda}\in E$ such
that $J_{\lambda}'(u_{\lambda})=0$ and
$J_{\lambda}(u_{\lambda})=c_{\lambda}.$ Then $u_{\lambda}$ is a
nontrivial solution of problem \eqref{e1.1}. Moreover, since
$u_{\lambda}$ is a critical point of $J_{\lambda}$, by $(M)$ and
$\gamma\in [p,p_{s}^{*}]$, we have
\begin{equation}
\begin{aligned}\label{29}
\kappa\lambda^{-\frac{p}{2p_{s,
\mu}^{*}-p}}&\geq J_{\lambda}(u_{\lambda})=J_{\lambda}(u_{\lambda})-\frac {1}{\gamma}J_{\lambda}'\left(u_{\lambda}\right)u_{\lambda}\\
&=\frac{1}{p} \widetilde{M}\left([u_{\lambda}]_{s, A}^{p}\right)-\frac {1}{\gamma}M\left([u_{\lambda}]_{s, A}^{p}\right)[u_{\lambda}]_{s, A}^{p}+\left(\frac {1}{p}-\frac {1}{\gamma}\right)\int_{R^{N}}V(x)|u_{\lambda}|^{p}\mathrm{d}x\\
&\quad +(\frac {1}{\gamma}-\frac {1}{2p_{s,\mu}^{*}})\lambda\iint_{\mathbb{R}^{2 N}} \frac{|u_{\lambda}(y)|^{p_{s, \mu}^{*}}|u_{\lambda}(x)|^{p_{s, \mu}^{*}}}{|x-y|^{\mu}} \mathrm{d} y \mathrm{~d} x+k\int_{R^{N}}\left[\frac {1}{\tau}|u_{\lambda}|^{q}-\frac {1}{q}|u_{\lambda}|^{q}\right]\mathrm{d}x\\
&\geq\left(\frac{\sigma}{p}-\frac{1}{\gamma}\right)m_{0}[u_{\lambda}]^{p}_{s,A}+\left(\frac{1}{p}-\frac{1}{\gamma}\right)\int_{R^{N}}V(x)|u_{\lambda}|^{p}\mathrm{d}x\\
&\quad +\left(\frac {1}{\gamma}-\frac
{1}{2p_{s,\mu}^{*}}\right)\lambda\iint_{\mathbb{R}^{2 N}}
\frac{|u_{\lambda}(y)|^{p_{s, \mu}^{*}}|u_{\lambda}(x)|^{p_{s,
\mu}^{*}}}{|x-y|^{\mu}} \mathrm{d} y \mathrm{~d} x+\left(\frac
{1}{\gamma}-\frac
{1}{q}\right)k\int_{R^{N}}|u_{\lambda}|^{q}\mathrm{d}x.
\end{aligned}
\end{equation}
So, we have $u_{\lambda} \rightarrow 0$ as $\lambda \rightarrow
\infty$. This completes the proof of Theorem 1.1.
\end{proof}
\begin{proof}[\bf Proof of Theorem 1.2]
Denote the set of all symmetric (in the sense that $-Z=Z$) and
closed subsets of $E$ by $\sum$.  For each $Z\in\sum$, define gen
$(Z)$ to be the Krasnoselski genus and
$$j(Z):=\min_{\iota\in\Gamma_{m^{*}}}gen(\iota(Z)\cap\partial B_{\varrho_{\lambda}}),$$
where $\Gamma_{m^{*}}$ is the set of all odd homeomorphisms
$\iota\in C(E,E)$, and $\varrho_{\lambda}$ is the number from Lemma
4.1. Then $j$ is a version of Benci's pseudoindex \cite{benci}. Let
$$c_{\lambda i}:=\inf_{j(Z)\geq i}\sup_{u\in Z}J_{\lambda}(u),1\leq i\leq m^{*}.$$
Since $J_{\lambda}(u)\geq\alpha_{\lambda}$ for all $u\in \partial
B^{+}_{\varrho_{\lambda}}$ and since $j(F_{\lambda m^{*}})=\dim
F_{\lambda m^{*}}=m^{*}$, we have
$$\alpha_{\lambda}\leq c_{\lambda1}\leq\cdot\cdot\cdot\leq c_{\lambda m^{*}}\leq\sup_{u\in H_{\lambda m^{*}}}J_{\lambda}(u)\leq\kappa\lambda^{-\frac{p}{2p_{s,
\mu}^{*}-p}}.$$  Lemma \ref{lem3.4} implies that $J_{\lambda}$
satisfies the $(PS)_{c_{\lambda}}$ condition at all levels
$c<\sigma_{0}\lambda^{-\frac{p}{2p_{s, \mu}^{*}-p}}.$ By the
critical critical point theory, we know that all $c_{\lambda i}$ are
critical levels, that is,  $J_{\lambda}$ has at least $m^{*}$ pairs
of nontrivial critical points satisfying
$$\alpha_{\lambda}\leq J_{\lambda}(u_{\lambda})\leq \kappa\lambda^{-\frac{p}{2p_{s,
\mu}^{*}-p}}.$$ Therefore, problem \eqref{e1.1} has at least $m^{*}$
pairs of solutions and $u_{\lambda,\pm i} \rightarrow 0$ as $\lambda
\rightarrow \infty$.
\end{proof}

\section*{Declarations}

\subsection*{Acknowledgments}
The authors thank the reviewers for their constructive remarks on
their work.

\subsection*{Funding}
Song was supported by the Science and Technology Development Plan
Project of Jilin Province, China (Grant No. 20230101287JC), the
National Natural Science Foundation of China (Grant No.12001061) and
Innovation and Entrepreneurship Talent Funding Project of Jilin
Province (No.2023QN21). Repov\v{s} was supported by the Slovenian
Research agency grants P1-0292, N1-0278, N1-0114, N1-0083, J1-4031,
and J1-4001.

\subsection*{Author contributions}
All authors contributed to the study conception, design, material
preparation, data collection, and analysis. All authors read and
approved the final manuscript.

\subsection*{Conflict of interest}
The authors state that there is no conflict of interest.

\subsection*{Ethical approval}
The conducted research is not related to either human or animal use.

\subsection*{Data availability statement} Data sharing is not applicable to this article as no data sets were generated or
analyzed during the current study.

\end{document}